\documentclass{amsart}
\usepackage{amssymb}
\usepackage{mathtools}
\usepackage[latin1]{inputenc}
\usepackage[swedish,english]{babel}
\usepackage{enumerate}
\usepackage{ifpdf}
\ifpdf
  \usepackage[pdftex]{graphicx}
  \usepackage{epstopdf}
\else
  \usepackage[dvips]{graphicx}
\fi
\usepackage[pdftitle={Split-step methods in jump SDEs},
  pdfauthor={Stefan Engblom},
  pdffitwindow=true,
  breaklinks=true,
  colorlinks=true,
  urlcolor=blue,
  linkcolor=red,
  citecolor=red,
  anchorcolor=red]{hyperref}
\usepackage[numbers,sort]{natbib}
\numberwithin{equation}{section}
\numberwithin{table}{section}
\numberwithin{figure}{section}

\theoremstyle{plain}
\newtheorem{theorem}{Theorem}[section]
\newtheorem{lemma}[theorem]{Lemma}

\theoremstyle{definition}
\newtheorem{definition}{Definition}[section]
\newtheorem{assumption}[definition]{Assumption}

\theoremstyle{remark}

\newcommand{\Realdom}{\mathbf{R}}
\newcommand{\Intdom}{\mathbf{Z}}

\newcommand{\Master}{\mathbb{M}^{T}}
\newcommand{\Masteradj}{\mathbb{M}}

\newcommand{\stoich}{\mathbb{N}}

\DeclareMathOperator{\Expect}{E}
\newcommand{\Probspace}{\Omega}
\newcommand{\Probelem}{\omega}
\newcommand{\Probfiltr}{\mathcal{F}}
\newcommand{\Prob}{\mathbf{P}}

\newcommand{\stopping}{P}
\newcommand{\taustopping}{\tau_{\stopping}}
\renewcommand{\stop}[1]{\hat{#1}}

\newcommand{\loc}{\mathrm{loc}}
\newcommand{\Sspace}[1]{S_{\Probfiltr}^{#1,\loc}(\Intdom_{+}^{D})}

\newcommand{\R}{\mathcal{R}}

\newcommand{\lvec}{\boldsymbol{l}}
\newcommand{\lvect}{\boldsymbol{l}^{T}}
\newcommand{\lnorm}[1]{\left\| #1 \right\|_{\lvec}}

\newcommand{\Vol}{V}
\newcommand{\fatmu}{\boldsymbol{\mu}}
\newcommand{\fatnu}{\boldsymbol{\nu}}
\newcommand{\X}{\boldsymbol{X}}
\newcommand{\stoichd}{\mathbb{S}}

\newcommand{\cadlag}{c{\`a}dl{\`a}g}

\begin{document}

\title[Split-step methods in jump SDEs]{Strong convergence for
  split-step methods in stochastic jump kinetics}

\author[S. Engblom]{Stefan Engblom}
\address{Division of Scientific Computing \\
  Department of Information Technology \\
  Uppsala University \\
  SE-751 05 Uppsala, Sweden.}
\urladdr{\url{http://user.it.uu.se/~stefane}}
\email{stefane@it.uu.se}
\thanks{Corresponding author: S. Engblom, telephone +46-18-471 27 54,
  fax +46-18-51 19 25.}


% 65-XX Numerical analysis
% 65Cxx Probabilistic methods, simulation and stochastic differential equations
% 65C40 Computational Markov chains
%
% 60-xx Probability theory and stochastic processes:
% 60Hxx Stochastic analysis
% 60H35 Computational methods for stochastic equations
%
% 60Jxx Markov processes:
% 60J27 Continuous-time Markov processes on discrete state spaces
% 60J28 Applications of continuous-time Markov processes on discrete
%       state spaces
% 60J75 Jump processes
%
% 92-xx Biology and other natural sciences:
% 92Cxx Physiological, cellular and medical topics:
% 92C42 Systems biology, networks
% 92C45 Kinetics in biochemical problems (pharmacokinetics, enzyme
%       kinetics, etc.)


\date{August 10, 2015}

\begin{abstract}
  Mesoscopic models in the reaction-diffusion framework have gained
  recognition as a viable approach to describing chemical processes in
  cell biology. The resulting computational problem is a
  continuous-time Markov chain on a discrete and typically very large
  state space. Due to the many temporal and spatial scales involved
  many different types of computationally more effective multiscale
  models have been proposed, typically coupling different types of
  descriptions within the Markov chain framework.

  In this work we look at the strong convergence properties of the
  basic first order Strang, or Lie-Trotter, split-step method, which
  is formed by decoupling the dynamics in finite time-steps. Thanks to
  its simplicity and flexibility, this approach has been tried in many
  different combinations.

  We develop explicit sufficient conditions for path-wise
  well-posedness and convergence of the method, including error
  estimates, and we illustrate our findings with numerical
  examples. In doing so, we also suggest a certain partition of unity
  representation for the split-step method, which in turn implies a
  concrete simulation algorithm under which trajectories may be
  compared in a path-wise sense.
\end{abstract}

\selectlanguage{english}

\maketitle

%**************************************************************************

\section{Introduction}

Since their introduction by Gillespie \cite{gillespie, gillespieCME},
stochastic models of chemical reactions have become ubiquitous tools
in describing the kinetics of living cells. Since complete Molecular
Dynamics-type descriptions of most biochemical processes are either
impractical or out of reach for complexity reasons, stochastic models
have remained popular as a viable alternative. Formulated in a way
which resembles the macroscopic viewpoint, but with randomness taking
certain microscopic effects into account, \emph{mesoscopic} stochastic
models attempt to strike a balance between computational feasibility
and accuracy. In fact, a common theme in several studies is the
discrepancy between deterministic and stochastic descriptions
\cite{circadian,noisygeneregul,newsteadystates_RDME}.

Due to the presence of multiple scales in species abundance and in
reaction rates, the computational problem of simulating well-stirred
or spatially extended models has caught a lot of attention. For
example, these features are the driving motivation behind the
development of \emph{hybrid methods}
\cite{haseltine_HSSA,kurtz_multiscale} and the various kinds of
\emph{model reduction} techniques that have been proposed \cite{QSSA2,
  smalltimesteps, nestedSSA, ssa_parareal}. Similarly, more efficient
time discretization ``tau-leap'' methods were proposed early on
\cite{tau_leap}, and has since then been modified and analyzed in
various ways \cite{postleap, tau_li, tau_leap_anal}.

As a means to facilitate multiscale- and multiphysics coupling in
method's development in general, \emph{split-step} methods have a long
story. Originally developed via (finite-dimensional) operator
splitting \cite{Strang}, in the present case these methods became
particularly important in the more computationally demanding spatial
stochastic reaction-diffusion setting \cite{master_spatial}. An
analysis in the sense of convergence in distribution in the master
equation setting was presented in \cite{jahnke_split}. A practical
method based on splitting to simulate fractional diffusion was
reported in \cite{fracdiff_split}, and an adaptive reaction-diffusion
simulator was suggested in \cite{hellander_split}. Finally, in
\cite{parallel_LKMC}, the splitting technique was used to bring out
parallelism from an otherwise strictly serial description.

In this work we look at the strong path-wise convergence of the basic
first order (in the operator sense) split-step method. A key issue
here is to devise a meaningful coupling between different trajectories
conditioned on different split-step discretizations. We solve this by
using a partition of unity representation which as a by-product also
implies a practical algorithm. Sufficient conditions for strong
convergence of order $1/2$ that apply notably also to \emph{open}
chemical systems are described and this is also confirmed in our
numerical experiments. An interesting observation is that, although
still only formally of strong order $1/2$, the second order (again in
the operator sense) Strang splitting performs considerably better than
the first order splitting.

In \S\ref{sec:jsdes} we recapitulate the description of chemical
processes as continuous-time Markov chains, in the non-spatial as well
as in the spatially extended case. Our main theoretical findings,
including explicit conditions for strong convergence, are reported in
\S\ref{sec:analysis}. Numerical illustrations are presented in
\S\ref{sec:experiments}, and a concluding discussion is found in
\S\ref{sec:conclusions}.

%**************************************************************************

\section{Stochastic jump kinetics}
\label{sec:jsdes}

We summarize in this section the mathematical background required in
the description of biochemical processes. A recapitulation of the
traditional well-stirred setting is found in
\S\ref{subsec:network}. Path-wise descriptions are found in
\S\ref{subsec:paths}, where some fundamental tools from stochastic
analysis are also reviewed. Finally, in \S\ref{subsec:spatiality} the
required extensions to encompass also spatially extended models are
indicated.

\subsection{Well-stirred Markovian reactions}
\label{subsec:network}

In a memory-less Markovian chemical system, at any instant $t$, the
\emph{state} is an integer vector $X(t) \in \Intdom_{+}^{D}$ counting
the number of molecules of each of $D$ species. The \emph{reactions}
are prescribed transitions of the state according to an intensity law,
or \emph{reaction propensity};
\begin{align}
  \label{eq:prop0}
  w_{r}: \Intdom_{+}^{D} \to \Realdom_{+}, \\
  \label{eq:prop}
  \Prob\left[X(t+dt) = x-\stoich_{r}| \; X(t) = x\right] &=
  w_{r}(x) \, dt+o(dt).
\end{align}
The system is thus fully described by the pair $[\stoich,w(x)]$, that
is, the \emph{stoichiometric matrix} $\stoich \in \Intdom^{D \times
  R}$, and $w(x) \equiv [w_{1}(x), \ldots, w_{R}(x)]^{T}$, the column
vector of reaction propensities. An important remark is that useful
physical descriptions are always \emph{conservative}, for all
propensities it holds that $w_{r}(x) = 0$ whenever $x-\stoich_{r} \not
\in \Intdom_{+}^{D}$ \cite[Chap.~8.2.2, Definition~2.4]{BremaudMC}.

The \emph{chemical master equation} (CME) \cite[Chap.~V]{VanKampen},
or \emph{Kolmogorov's forward differential system}
\cite[Chap.~8.3]{BremaudMC} governs the law of the state $X(t)$
conditioned on some initial state. Put $p(x,t) = \Prob(X(t) = x | \;
X(0) = x_{0})$. Then
\begin{align}
  \label{eq:CME}
  \frac{\partial p(x,t)}{\partial t} &= 
  \sum_{r = 1}^{R} w_{r}(x+\stoich_{r})p(x+\stoich_{r},t)-
  w_{r}(x)p(x,t) =: \Master p(x,t),
\end{align}
where $\Masteradj$ is the \emph{infinitesimal generator} of the
process.

\subsection{Path-wise representations}
\label{subsec:paths}

The use of path-wise representations in the analysis of Markov
processes on discrete state-spaces in continuous time was pioneered by
Kurtz in a series of paper (see \cite{Markovappr} and the references
therein). We thus postulate the probability space
$(\Probspace,\Probfiltr,\Prob)$, where the filtration $\Probfiltr_{t
  \ge 0}$ contains $R$-dimensional Poisson processes. The transition
law \eqref{eq:prop} implies a certain counting process which can be
constructed from a standard unit-rate Poisson process $\Pi_{r}$. The
state $X(t)$ can then be written
\begin{align}
  \label{eq:Poissrepr}
  X_{t} &= X_{0}-\sum_{r = 1}^{R} \stoich_{r} \Pi_{r}
  \left(  \int_{0}^{t} w_{r}(X_{s-}) \, ds \right).
\end{align}
This is Kurtz's \emph{random time change representation}
\cite[Chap.~6.2]{Markovappr} which gives rise to the notion of
\emph{operational time} in the argument to each of the $R$ independent
Poisson processes. Note that, in \eqref{eq:Poissrepr}, by $X(t-)$ is
meant the state before any transitions at time $t$.

It is sometimes convenient to use an equivalent construction in terms
of a random counting measure \cite[Chap.~VIII]{pointPQ}. We denote by
$\mu_{r}(dt) = \mu_{r}(w_{r}(X_{t-}) \, dt; \, \Probelem)$ for
$\Probelem \in \Probspace$ the random measure associated with the
counting process whose intensity at any instant $t$ is
$w_r(X_{t-})$. Thus with deterministic intensity $\Expect[\mu_{r}(dt)]
= \Expect[w_{r}(X_{t-}) \, dt]$, this defines an increasing sequence
of exponentially distributed counts $\tau_{i} \in \Realdom_{+}$. With
$\fatmu = [\mu_{1},\ldots,\mu_{R}]^{T}$ we can now write
\eqref{eq:Poissrepr} in the compact differential form
\begin{align}
  \label{eq:SDE}
  dX_{t} &= -\stoich \fatmu(dt).
\end{align}

For realistic chemical systems the number of molecules must somehow be
bound \textit{a priori}. We encapsulate this property by requiring the
existence of a certain weighted norm
\begin{align}
  \label{eq:ldef}
  \lnorm{x} := \lvect x, \quad x \in \Intdom_{+}^{D},
\end{align}
normalized such that $\min_{i} \lvec_{i} = 1$. Equipped with this norm
we formulate, following \cite{jsdestab} closely (see also
\cite{ergodic_jsde, momentbound_jsde}),

\begin{assumption}
  \label{ass:ass}
  For arguments $x$, $y \in \Intdom_{+}^{D}$ we assume that
  \begin{enumerate}[(i)]
  \item \label{it:1bnd} $-\lvect \stoich w(x) \le A+\alpha \lnorm{x}$,
   
  \item \label{it:bnd} $(-\lvect \stoich)^{2} w(x)/2 \le B+\beta_{1}
    \lnorm{x}+\beta_{2}\lnorm{x}^{2}$,

  \item
    \label{it:lip} $|w_{r}(x)-w_{r}(y)| \le L_{r}(\stopping)\|x-y\|$, for $r =
    1,\ldots,R$, and $\lnorm{x} \vee \lnorm{y} \le \stopping$.
  \end{enumerate}
  With the exception of $\alpha$, all parameters
  $\{A,B,\beta_{1},\beta_{2},L\}$ are assumed to be non-negative.
\end{assumption}

In order to state an \textit{a priori} result concerning the
regularity of the solutions to \eqref{eq:SDE}, following
\cite[Sect.~3.1.2]{jumpSDEs}, we define the following family of spaces
of path-wise locally bounded processes:
\begin{align}
  \Sspace{p} &= \left\{ X(t,\Probelem): \begin{array}{l}
    X_{t} \in \Intdom_{+}^{D} \mbox{ is
      $\Probfiltr_{t}$-adapted such that } \\
    \Expect \, \sup_{t \in [0,T]} \lnorm{X_{t}}^{p} <
    \infty \mbox{ for } \forall T < \infty \end{array} \right\}.
\end{align}

\begin{theorem}[\textit{Theorem~4.7 in \cite{jsdestab}}]
  \label{th:exist0}
  Let $X_{t}$ be a solution to \eqref{eq:SDE} under
  Assumption~\ref{ass:ass}~\eqref{it:1bnd} and \eqref{it:bnd} with
  $\beta_{2} = 0$. Then if $\lnorm{X_{0}}^{p} < \infty$, $\{X_{t}\}_{t
    \ge 0} \in \Sspace{p}$. If $\beta_{2} > 0$ then the conclusion
  remains under the additional requirement that $\lnorm{X_{0}}^{p+1} <
  \infty$.
\end{theorem}

Below we will frequently use the stopping time
\begin{align}
  \label{eq:stopping}
  \taustopping := \inf_{t \ge 0}\{\lnorm{X_{t}} > \stopping\}
\end{align}
and put $\stop{t} = t \wedge \taustopping$ for some finite $t$
defining an interval of interest. As an example, a differential form
of It\^o's change of variables formula can be derived formally by
simply summing over jump times \cite[Chap.~4.4.2]{LevySDEs}
\begin{align}
  \label{eq:Ito}
  df(X_{t}) &= \sum_{r = 1}^{R} f(X_{t-}-\stoich_{r})-
  f(X_{t-}) \, \mu_{r}(dt).
\end{align}
More carefully, \emph{Dynkin's formula} for the stopped process is
then given by \cite[Chap.~9.2.2]{BremaudMC},
\begin{align}
  \label{eq:Dynkin}
  \Expect f(X_{\stop{t}})-\Expect f(X_{0}) &= 
  \int_{0}^{\stop{t}} \sum_{r = 1}^{R} \Expect \left[ (f(X_{s}-\stoich_{r})-
  f(X_{s})) w_{r}(X_{s}) \right] \, ds.
\end{align}

In order to efficiently work with the Poisson representation
\eqref{eq:Poissrepr} in the sense of mean square, the following two
lemmas which follows \cite[\textit{Manuscript}]{jsdevarsplit} very
closely will be critical.
\begin{lemma}
  \label{lem:PiStop}
  Let $\Pi$ be a unit-rate Poisson process and $T$ a bounded stopping
  time, both adapted to $\Probfiltr_{t}$. Then
  \begin{align}
    \label{eq:PiStop1}
    \Expect[\Pi(T)] &= \Expect[T], \\
    \label{eq:PiStop2}
    \Expect[\Pi^2(T)] &= 2\Expect[\Pi(T)T]-\Expect[T^{2}]+\Expect[T].
  \end{align}
\end{lemma}

\begin{proof}
  Let $\tilde{\Pi}(t) := \Pi(t)-t$ be the compensated process. This is
  a martingale and Doob's optional sampling theorem implies
  $\Expect[\tilde{\Pi}(T)] = 0$ \cite[Theorem 17,
    Chap.~I.2]{protterSDE}, which is \eqref{eq:PiStop1}. Similarly
  $Z(t) := \tilde{\Pi}^2(t)-t$ is a martingale \cite[Theorem 24,
    Chap.~I.3]{protterSDE} and the sampling theorem yields
  $\Expect[Z(T)] = 0$, or,
  \begin{align*}
    0 &= \Expect[\Pi^2(T)-2\Pi(T)T+T^{2}-T],
  \end{align*}
  which is \eqref{eq:PiStop2}.
\end{proof}

\begin{lemma}
  \label{lem:Stopping}
  Let $\Pi$ be a unit-rate Poisson process and $T_1$, $T_2$ bounded
  stopping times, all adapted to $\Probfiltr_{t}$. Then
  \begin{align}
    \label{eq:Stopping1}
    \Expect[|\Pi(T_2)-\Pi(T_1)|] &= \Expect[|T_2-T_1|], \\
    \label{eq:Stopping2}
    \Expect[(\Pi(T_2)-\Pi(T_1))^2] &=
    2\Expect[|\Pi(T_2)-\Pi(T_1)|(T_1 \vee T_2)] \\
    \nonumber
    &\hphantom{=} -\Expect[|T_2^{2}-T_1^{2}|]+\Expect[|T_2-T_1]].
  \end{align}
\end{lemma}

The formulation \eqref{eq:Stopping1} was recently used in
\cite[\textit{Manuscript}]{AGangulyMMS2014} to provide for a related
analysis in the sense of convergence in mean.

\begin{proof}
  Assume first that $T_2 \ge T_1$. By
  Lemma~\ref{lem:PiStop}~\eqref{eq:PiStop1},
  \begin{align*}
    \Expect[\Pi(T_2)-\Pi(T_1)] &= \Expect[T_2-T_1].
  \end{align*}
  For general stopping times $S_1$, $S_2$, say, \eqref{eq:Stopping1}
  follows upon substituting $T_1 := S_1 \wedge S_2$ and $T_2 := S_1
  \vee S_2$ into this equality.

  Next put $X := \Expect[(\Pi(T_2)-\Pi(T_1))^2]$ and assume anew that
  $T_2 \ge T_1$. We find
  \begin{align*}
    X &= \Expect[\Pi(T_2)^2 +
        \Pi(T_1)^2 - 2\Pi(T_1)\Pi(T_2)] \\
      &= \Expect[\Pi(T_2)^2 + \Pi(T_1)^2] -
        2\Expect[\Pi(T_1)\Expect[\Pi(T_2)|\Probfiltr_{T_1}]]. \\
    \intertext{To evaluate the iterated expectation, note that}
    \Expect[\tilde{\Pi}(T_2)|\Probfiltr_{T_1}] &= \tilde{\Pi}(T_1)
    \Longrightarrow
    \Expect[\Pi(T_2)|\Probfiltr_{T_1}] =
    \Pi(T_1)-T_1+E[T_2|\Probfiltr_{T_1}]. \\
    \intertext{Hence}
    \Expect[\Pi(T_1)\Expect[\Pi(T_2)|\Probfiltr_{T_1}]] &= 
    \Expect[\Pi(T_1)^2]-
    \Expect[\Pi(T_1)T_1]+\Expect[\Pi(T_1)T_2], \\
    \intertext{and so}
    X &= \Expect[\Pi(T_2)^2 - \Pi(T_1)^2]
        +2\Expect[\Pi(T_1)T_1]-2\Expect[\Pi(T_1)T_2]. \\
    \intertext{Applying Lemma~\ref{lem:PiStop}~\eqref{eq:PiStop2}
      twice then yields}
    X &= 2\Expect[(\Pi(T_2)-\Pi(T_1))T_2]-\Expect[T_2^{2}-T_1^{2}]+
    \Expect[T_2-T_1].
  \end{align*}
  As before the substitutions $T_1 := S_1 \wedge S_2$ and $T_2 := S_1
  \vee S_2$ implies \eqref{eq:Stopping2} for general stopping times
  $S_1$, $S_2$.
\end{proof}

Lemma~\ref{lem:Stopping} will be applied as follows. Assuming first
the \textit{a priori} bound $T_1 \vee T_2 \le B$ we get from
\eqref{eq:Stopping1}--\eqref{eq:Stopping2} that
\begin{align}
  \label{eq:finalStop1}
  \Expect[(\Pi(T_2)-\Pi(T_1))^2] &\le (2B+1)\Expect[|T_2-T_1|].
\end{align}
Let $\Probfiltr_t$ be the filtration adapted to
$\{\tilde{\Pi}_r\}_{r = 1}^R$. Then for any fixed time $t$, $T_r(t)
:= \int_{0}^{t} w_r(X(s)) \, ds$ is a stopping time adapted to
\cite[Lemma~3.1]{tauAnderson}
\begin{align*}
  \tilde{\Probfiltr}^r_u := \sigma\{\Pi_r(s), s \in [0,u];\;
  \Pi_{k \not = r}(s), s\in [0,\infty]\}.
\end{align*}
Intuitively, since $X(t) = \sum_r \Pi_r(T_r(t)) \stoich_r$, the event
$\{T_r(t) < u\}$ depends on $\Pi_r$ during $[0,u]$ and on $\{\Pi_k, k
= 1\ldots R,\,k\neq r\}$ during $[0,\infty)$. However, since
  $\{\Pi_r\}_{r = 1}^R$ are all independent, $\tilde{\Pi}_r$ is still
  a martingale with respect to $\tilde{\Probfiltr}^r_u$ (and not only
  with respect to $\Probfiltr^r_u = \sigma\{\Pi_r(s), s \in [0,u]
  \}$). Hence we can apply the stopping time theorems to $T_r(t)$ and
  the previous lemmas apply.

For an approximating process $\tilde{X} \approx X$, say, assuming a
suitable random time representation in the form of
\eqref{eq:Poissrepr} is available, these results will remain valid for
$\tilde{X}$ as well. To conclude, given the bound
\begin{align}
  \int_{0}^{t}w_r(X(s)) \, ds \vee 
  \int_{0}^{t}\tilde{w}_r(\tilde{X}(s)) \, ds \le B.
\end{align}
we get from \eqref{eq:finalStop1} that
\begin{align}
  \nonumber
  \Expect\left[\left(\Pi_r\left(\int_{0}^{t}w_r(X(s)) \, ds\right)-
    \Pi_r\left(\int_{0}^{t}\tilde{w}_r(\tilde{X}(s)) \, ds \right)
    \right)^2\right] \\
  \label{eq:finalStop}
  \le
  (2B+1)\Expect\left[\left|\int_{0}^{t}w_r(X(s)) \, ds-
    \int_{0}^{t}\tilde{w}_r(\tilde{X}(s)) \, ds\right|\right].
\end{align}

\subsection{Incorporating spatial dependence}
\label{subsec:spatiality}

Well-stirred modeling of chemical kinetics relies on
\emph{homogeneity}, that is, that the probability of finding a
molecule is equal throughout the volume. There are many situations of
interest where this assumption is violated, for instance, when slow
molecular transport allows concentration gradients to build up. A way
to approach this situation is through compartmentalization techniques
\cite[Chap.~XIV]{VanKampen}, which leads to models with a very large
number of generalized reaction channels. Since split-step methods are
a particularly promising computational technique here, we briefly
review this framework.

The basic premise is that, although the full volume $\Vol$ is not
well-stirred, it can be subdivided into smaller cells $\Vol_{j}$ such
that their individual volume $|\Vol_{j}|$ is small enough that
diffusion suffices to make each cell practically well-stirred.

The state of the system is thus now an array $\X \in \Intdom_{+}^{D
  \times K}$ consisting of $D$ chemically active species $\X_{ij}$, $i
= 1,\ldots,D$, in $K$ cells, $j = 1,\ldots,K$. This state is changed
by chemical reactions occurring in each cell (vertically in $\X$) and
by diffusion where molecules move to adjacent cells (horizontally in
$\X$).

Since each cell is assumed to be well-stirred, \eqref{eq:SDE} governs
the reaction dynamics,
\begin{align}
  \label{eq:SDEr}
  d\X_{t} &= -\stoich \fatmu(dt),
\end{align}
where $\fatmu$ is now $R$-by-$K$ with $\Expect[\mu_{rj}] =
\Expect[w_{rj}(\X_{(\cdot,j)}(t-)) \, dt]$. Transport of a molecule
from $\Vol_{k}$ to $\Vol_{j}$ can also be thought of as a special kind
of reaction,
\begin{align}
  \label{eq:dprop}
  \X_{ik} &\xrightarrow{q_{kji} \X_{ik}} \X_{ij},
\end{align}
where the rate constant $q_{kji}$ is non-zero only for connected cells.
%
%% For diffusion, as a consistent approximation when the discretization
%% becomes finer, $q_{kj}$ should be taken as the inverse of the mean
%% first exit time for a single molecule of species $i$ undergoing
%% Brownian motion from cell $\Vol_{k}$ to $\Vol_{j}$.
%
In practice, for any given spatial discretization, numerical methods
may be used to define the diffusion rates consistently
\cite{master_spatial}. We obtain from \eqref{eq:dprop} the mesoscopic
diffusion model
\begin{align}
  \label{eq:SDEd}
  d\X_{t} &= \stoichd \otimes (-\fatnu^{T}+\fatnu)(dt),
\end{align}
where $\stoichd \in \Intdom^{1 \times K}$ of all 1's, where $\fatnu$
is $K$-by-$K$-by-$D$ with $\Expect[\nu_{kji}] = \Expect[q_{kji}
  \X_{ik}(t-) \, dt]$, and where the array operations are suitably
defined. In \eqref{eq:SDEd}, note how diffusion \emph{exit} events are
paired with \emph{entry} events via the terms $-\fatnu^{T}$ and
$\fatnu$, respectively.

By superposition of \eqref{eq:SDEr} and \eqref{eq:SDEd} we arrive at
the reaction-diffusion model
\begin{align}
  \label{eq:RDME}
  d\X_{t} &= -\stoich \fatmu(dt)+\stoichd \otimes (-\fatnu^{T}+\fatnu)(dt).
\end{align}
As was already noted in \cite{master_spatial}, part of the interest in
split-step methods comes from simulating reactions and diffusions by
different methods. For simplicity, in the rest of the paper we shall
take the well-stirred case \eqref{eq:SDE} as our target of study. In
doing so we keep in mind that the reaction-diffusion (or
reaction-transport) case \eqref{eq:RDME} does fall under the same
general class of descriptions.

%**************************************************************************

\section{Analysis of split-step methods}
\label{sec:analysis}

In this section we present our main theoretical findings. The
splitting we choose to analyze is defined in the master equation
setting in \S\ref{subsec:cme_split}. In order to couple trajectories
and obtain path-wise comparisons, the splitting is redefined in the
operational time framework in \S\ref{subsec:optime_split}, where some
\textit{a priori} estimates are also derived. After developing a few
further preliminary results in \S\ref{subsec:lemmas}, the theory is
put together in \S\ref{subsec:convergence}, where our main convergence
result is presented.

Throughout this section we let $C$ denote a positive constant which
may be different at each occurrence.

\subsection{Operator splitting and the master equation}
\label{subsec:cme_split}

While we take another approach below, traditionally, split-step
methods are constructed via operator-splitting of the master equation
\eqref{eq:CME}. Assume the split into two sets of reaction pathways
can be written as
\begin{align}
  \label{eq:split1}
  \stoich &= \left[ \stoich^{(1)} \; \stoich^{(2)} \right], \\
  \label{eq:split2}
  w(x) &= \left[ w^{(1)}(x); \; w^{(2)}(x) \right],
\end{align}
where $\stoich^{(i)}$ is $D$-by-$R_{i}$, $i \in \{1,2\}$, $R_{1}+R_{2}
= R$, and where the propensity column vectors have the corresponding
dimensions. The simplest possible split-step method, and the one we
choose to analyze in this paper, can then be written in integral form
as (compare \eqref{eq:CME})
\begin{align}
  \label{eq:CMEsplit1}
  \tilde{p}_h(x,t+h) &= p_h(x,t)+\int_{t}^{t+h} 
  \sum_{r \in \R_{1}} w_{r}(x+\stoich_{r})\tilde{p}_h(x+\stoich_{r},s)-
  w_{r}(x)\tilde{p}_h(x,s) \, ds, \\
  \label{eq:CMEsplit2}
  p_h(x,t+h) &= \tilde{p}_h(x,t+h)+\int_{t}^{t+h} 
  \sum_{r \in \R_{2}} w_{r}(x+\stoich_{r})p_h(x+\stoich_{r},s)-
  w_{r}(x)p_h(x,s) \, ds,
\end{align}
where $\R_{1} = \{1, \ldots, R_{1}\}$, $\R_{2} = \{R_{1}+1,\ldots,
R\}$. Loosely speaking, \eqref{eq:CMEsplit1} evolves the dynamics of
the first set of reactions in an auxiliary variable $\tilde{p}_h$ from
time $t$ to $t+h$, and \eqref{eq:CMEsplit2} similarly evolves the
second.

\subsection{Splitting in operational time}
\label{subsec:optime_split}

To obtain a concrete path-wise formulation which is more amenable to
analysis we first define the kernel step function
\begin{align}
  \sigma_{h}(t) &= 1-2 \left( \lfloor t/(h/2) \rfloor
  \mbox{{\footnotesize  mod }} 2 \right),
\end{align}
for convenience also visualized in Figure~\ref{fig:sigma}. This is a
piecewise constant \emph{\cadlag} function which may be used to
introduce `switching' events into the process that does not affect the
state but turns on or off selected parts of the dynamics. More
precisely, the split-step method
\eqref{eq:CMEsplit1}--\eqref{eq:CMEsplit2} for \eqref{eq:Poissrepr}
can be written in the operational time form
\begin{align}
  \label{eq:splitSDE}
  Y_{t} &= Y_{0}-\sum_{r \in \R_{1}} \stoich_{r} \Pi_{r} \left(
      \int_{0}^{t} (1+\sigma_{h}(s))w_{r}(Y_{s-}) \,
      ds \right) \\
    \nonumber
    &\phantom{= Y_{0}} -\sum_{r \in \R_{2}} \stoich_{r} \Pi_{r} \left(
      \int_{0}^{t} (1-\sigma_{h}(s))w_{r}(Y_{s-}) \,
      ds \right).
\end{align}
For convenience here and below we shall suppress the dependency on the
split-step length $h$; we simply write $Y(t)$ (or $Y_{t}$) instead of
$Y_h(t)$.

Eq.~\eqref{eq:splitSDE} is a partition of unity representation in
that, at any instant in \eqref{eq:splitSDE}, one of the sets of
reactions is turned off while the other operates at twice the
intensity. Since the length of each interval where the same set of
reactions is active is $h/2$, effectively the unit time for those
channels is evolved in steps of length $h$, in agreement with
\eqref{eq:CMEsplit1}--\eqref{eq:CMEsplit2}. In
\S\ref{subsubsec:Strang} below we show that the same type of
representation may be used when analyzing also the second order Strang
split method.

The main advantage with \eqref{eq:splitSDE} over
\eqref{eq:CMEsplit1}--\eqref{eq:CMEsplit2} is that the former may be
compared path-wise to \eqref{eq:Poissrepr}. Indeed, the convergence
results in \S\ref{subsec:convergence} concerns the behavior of
$\Expect\|(X-Y)(t)\|^{2}$ as the split-step $h \to 0$, where $X(t)$,
$Y(t)$ are solutions to \eqref{eq:Poissrepr} and \eqref{eq:splitSDE},
respectively. The approach to coupling processes via the random time
change representation was first used by Kurtz \cite{KurtzApprox} and
practically implies the \emph{Common Reaction Path} (CRP) method for
simulating coupled processes \cite{sensitivitySSA,aem_proceeding} (see
also \S\ref{subsec:implementation} below).

\begin{figure}[htp]
  \includegraphics{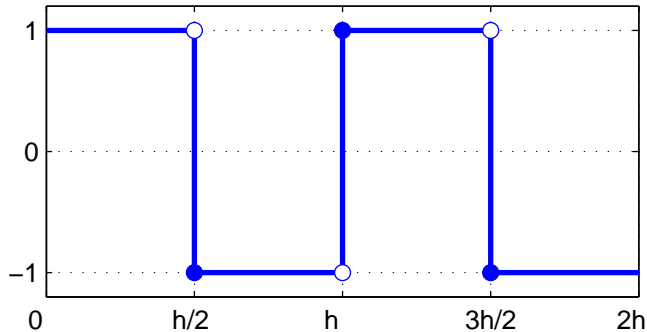}
  \caption{Definition of the piecewise constant \cadlag\ kernel
    function $\sigma_{h}(x)$.}
  \label{fig:sigma}
\end{figure}

\begin{assumption}
  \label{ass:split}
  In the following, our working assumptions will be that
  Assumption~\ref{ass:ass} holds for \emph{both} sub-systems
  $[\stoich^{(i)},w^{(i)}(x)]$, $i \in \{1,2\}$ in
  \eqref{eq:split1}--\eqref{eq:split2} and with \emph{the same}
  weight-vector $\lvec$. We separate the constants of the two
  sub-systems by using superscripts, as in $A^{(i)}$, $i \in \{1,2\}$,
  and additionally define $A^{(0)} := A^{(1)} \vee A^{(2)}$.
\end{assumption}

The assumption that the weight-vector $\lvec$ is the same for both
sub-systems as well as for the original description
\eqref{eq:Poissrepr} is mainly for convenience as it avoids switching
back and forth between equivalent norms. A further comment is that, in
view of a finite time-step $h$ it makes sense to require that both
sub-systems are well-posed in the sense of
Theorem~\ref{th:exist0}. However, one can rightly ask if this is
really necessary as $h \to 0$; for $h$ small enough finite-time
explosions are likely not going to be a problem. On balance we chose
to settle with the current \emph{sufficient} conditions as a complete
theory likely must contain several special cases (for an illustration,
see the numerical example in \S\ref{subsec:illposed}).

\begin{theorem}[\textit{Moment bound}]
  \label{th:Epbound}
  Let $Y(t)$ satisfy \eqref{eq:splitSDE} under
  Assumption~\ref{ass:split}. Then for any integer $p \ge 1$,
  \begin{align}
    \label{eq:Epbound}
    \Expect\lnorm{Y_{t}}^{p} &\le (\lnorm{Y_{0}}^{p}+1) \exp (Ct)-1,
  \end{align}
  with $C > 0$ a constant which depends on $p$ and on the relevant
  constants of the assumptions, but not on the split-step $h$.
\end{theorem}

It will be convenient to quote the following basic inequality.

\begin{lemma}[\textit{Lemma~4.6 in \cite{jsdestab}}]
  \label{lem:diffbound}
  Let $H(x) \equiv (x+y)^{p}-x^{p}$ with $x \in \Realdom_{+}$ and $y
  \in \Realdom$. Then for integer $p \ge 1$ we have the bounds
  \begin{align}
    \label{eq:signed_diffbound}
    H(x) &\le p y x^{p-1}+
    2^{p-4} p(p-1) y^{2} \left[ x^{p-2}+
      |y|^{p-2} \right], \\
   \label{eq:abs_diffbound}
    |H(x)| &\le p |y| 2^{p-2} \left[ x^{p-1}+
      |y|^{p-1} \right].
  \end{align}
\end{lemma}

\begin{proof}[\textit{Proof of Theorem~\ref{th:Epbound}}]
  The starting point is Dynkin's formula \eqref{eq:Dynkin} under the
  stopping time on $\lnorm{Y_{t}}$ defined in
  \eqref{eq:stopping}. With $f(x) = \lnorm{x}^{p} = [\lvect x]^{p}$ we
  get
  \begin{align}
    \nonumber
    \Expect\lnorm{Y_{\stop{t}}}^{p} &= \lnorm{Y_{0}}^{p}+
    \Expect\int_{0}^{\stop{t}} (1+\sigma_{h}(s))
    \overbrace{\sum_{r \in \R_{1}} w_{r}(Y_{s}) \left[ 
      \left[ \lvect (Y_{s}-\stoich_{r}) \right]^{p}-
      \left[ \lvect Y_{s} \right]^{p} \right]}^{=: G_{1}(Y_{s})} \, ds \\
    \label{eq:Gdef}
    &\hphantom{= \lnorm{Y_{0}}^{p}}
    +\Expect\int_{0}^{\stop{t}} (1-\sigma_{h}(s))
    \underbrace{\sum_{r \in \R_{2}} w_{r}(Y_{s}) \left[ 
      \left[ \lvect (Y_{s}-\stoich_{r}) \right]^{p}-
      \left[ \lvect Y_{s} \right]^{p} \right]}_{=: G_{2}(Y_{s})} \, ds.
  \end{align}
  Using the assumptions on the first sub-system
  $[\stoich^{(1)},w^{(1)}(x)]$ and the first part of
  Lemma~\ref{lem:diffbound} we obtain the bound
  \begin{align*}
    G_{1}(y) &\le p(A^{(1)}+\alpha^{(1)}\lnorm{y}) \lnorm{y}^{p-1}+
    \\ &\phantom{\le} 2^{p-3} p(p-1)
    (B^{(1)}+\beta_{1}^{(1)}\lnorm{y}+\beta_{2}^{(1)}\lnorm{y}^{2})
    (\lnorm{y}^{p-2}+\delta^{p-2}) \\
    &\le \frac{C}{2} \left( 1+\lnorm{y}^{p} \right),
  \end{align*}
  say, in which $\delta := \|\lvect \stoich^{(1)}\|_{\infty}$ and
  where Young's inequality was used several times to arrive at the
  second bound. A similar bound is readily found for $G_{2}$ so
  summing up from \eqref{eq:Gdef} we get
  \begin{align}
    \label{eq:momest}
    \Expect\lnorm{Y_{\stop{t}}}^{p} &\le \lnorm{Y_{0}}^{p}+
    \int_{0}^{t} C(1+\Expect\lnorm{Y_{\stop{s}}}^{p}) \, ds,
  \end{align}
  where $\stop{s} = s \wedge \taustopping$. By Grönwall's inequality
  this implies the bound
  \begin{align}
    \Expect\lnorm{Y_{\stop{t}}}^{p} &\le (\lnorm{Y_{0}}^{p}+1) \exp (Ct)-1,
  \end{align}
  which is independent of $\stopping$. We therefore arrive at
  \eqref{eq:Epbound} by letting $\stopping \to \infty$ and using
  Fatou's lemma.
\end{proof}

% (1) Firstly, for any \Probelem \in \Probspace, X(s,\omega) either
% explodes for some 0 \le s \le t, or is bounded by some value
% \stopping_0. In the former case we have that X(\stop{s},\omega) is
% unbounded as \stopping \to \infty. In the latter case,
% \lim_{\stopping \to \infty} X(\stop{s},\omega) = X(s,\omega) for 0
% \le s \le t.

% (2) Since \Expect X(\stop{t}) is bounded from above independently of
% \stopping (this is the bound derived previously) we must have that
% \lim_{\stopping \to \infty} X(\stop{t}) = X(t) almost surely.

% (3) Now for a convergent sequence, the limit and liminf are equal,
% hence \lim_{\stopping \to \infty} X(\stop{t}) = \liminf_{\stopping
% \to \infty} X(\stop{t}) almost surely, and so we have \Expect
% \lim_{\stopping \to \infty} X(\stop{t}) = \Expect \liminf_{\stopping
% \to \infty} X(\stop{t}), where Fatou finally applies.

Recall that the quadratic variation of a process $(X_{t})_{t \ge 0}$
in $\Realdom^{D}$ can be defined by (convergence in probability)
\begin{align}
  [X]_{t} &= \lim_{\|M\| \to 0} \sum_{k \in M}
  \left\| X_{t_{k+1}}-X_{t_{k}} \right\|^{2}, \\
  \intertext{where the mesh $M = \{0 = t_{0} < t_{1} < \cdots < t_{n}
    = t\}$ and where $\|M\| = \max_{k} |t_{k+1}-t_{k}|$. Similarly, we
    define for later use also the total variation}
  V_{[0,t]}(X) &= \lim_{\|M\| \to 0} \sum_{k \in M}
  \left\| X_{t_{k+1}}-X_{t_{k}} \right\|.
\end{align}

\begin{lemma}
  \label{lem:quadratic_variation}
  Let $Y(t)$ satisfy \eqref{eq:splitSDE} under
  Assumption~\ref{ass:split}. Then the quadratic variation of
  $\lnorm{Y_{t}}^{p}$ is bounded by
  \begin{align}
    \label{eq:quadratic_variation}
    \Expect[\lnorm{Y}^{p}]_{t}^{1/2} &\le \Expect \int_{0}^{t}
    C(1+\lnorm{Y_{s}}^{p}+ \beta_{2}^{(0)}\lnorm{Y_{s}}^{p+1}) \, ds,
  \end{align}
  where $C > 0$ again is independent on the split-step $h$ and where
  $\beta_{2}^{(0)} := \beta_{2}^{(1)} \vee \beta_{2}^{(2)}$.
\end{lemma}

\begin{proof}
  Instead of as in \eqref{eq:Gdef}, for brevity we shall use the
  following compact notation for sums involved in the two sub-systems
  which form the split-step method,
  \begin{align}
    \label{eq:split_sum}
    \sum_{r \in \R_{1},\R_{2}} (1\pm \sigma_{h}(s)) F(r) &:=
    \sum_{r \in \R_{1}} (1+\sigma_{h}(s)) F(r)+
    \sum_{r \in \R_{2}} (1-\sigma_{h}(s)) F(r).
  \end{align}
  Keeping this in mind we have
  \begin{align}
    \Expect \, [\lnorm{Y}^{p}]_{\stop{t}}^{1/2} &=
    \Expect \, \left[ \left( \int_{0}^{\stop{t}} 
      \sum_{r \in \R_{1},\R_{2}} \left( 
      \left[ \lvect (Y_{s}-\stoich_{r}) \right]^{p}-
      \left[ \lvect Y_{s} \right]^{p} \right)^{2} \, \mu_{r}(ds) \right)^{1/2} 
      \right].
  \end{align}
  Writing $\mu_r(dt) = (1\pm\sigma_{h}(t))
  w_r(Y_{t-})\,dt+\tilde{\mu}_r(dt)$, and from the inequality
  $\|\cdot\| \le \|\cdot\|_1$ we get after using that the random
  measure compensated with the deterministic intensity is a local
  martingale,
  \begin{align}
    &\le
    \Expect \left[ \int_{0}^{\stop{t}} 
      \sum_{r \in \R_{1},\R_{2}} (1\pm\sigma_{h}(s)) w_{r}(Y_{s}) \left| 
      \left[ \lvect (Y_{s}-\stoich_{r}) \right]^{p}-
      \left[ \lvect Y_{s} \right]^{p} \right| \, ds \right], \\
    \intertext{or, after using 
      Lemma~\ref{lem:diffbound}~\eqref{eq:abs_diffbound},}
    \nonumber
    &\le \Expect \left[ \int_{0}^{\stop{t}} 
      \sum_{r} (1\pm\sigma_{h}(s)) \, 
      p |\lvect \stoich_{r}| w_{r}(Y_{s}) \; 2^{p-2}
      \left[ \lnorm{Y_{s}}^{p-1}+|\lvect \stoich_{r}|^{p-1} \right]
      \, ds \right] \\
    \label{eq:special}
    &\le \Expect \left[ \int_{0}^{\stop{t}} 
      C(B^{(0)}+\beta_{1}^{(0)}\lnorm{Y_{s}}+\beta_{2}^{(0)}\lnorm{Y_{s}}^{2})
      \left[ \lnorm{Y_{s}}^{p-1}+\delta^{p-1} \right]
      \, ds \right]
  \end{align}
  by Assumption~\ref{ass:ass}~\eqref{it:bnd}, where $\delta = \|\lvect
  \stoich\|_{\infty}$. Using Theorem~\ref{th:Epbound} and letting
  $\stopping \to \infty$ we arrive at the stated bound.
\end{proof}

\begin{theorem}
  \label{th:exist}
  Let $Y_{t}$ satisfy \eqref{eq:splitSDE} under
  Assumption~\ref{ass:split} with $0 = \beta_{2}^{(0)} :=
  \beta_{2}^{(1)} \vee \beta_{2}^{(2)}$. Then if $\lnorm{Y_{0}}^{p} <
  \infty$, $\{Y_{t}\}_{t \ge 0} \in \Sspace{p}$ for all $h > 0$. If
  $\beta_{2}^{(0)} > 0$, then the conclusion remains under the
  additional requirement that $\lnorm{Y_{0}}^{p+1} < \infty$.
\end{theorem}

Note that the somewhat technical details concerning the case
$\beta_{2}^{(0)} \not = 0$ are shared with the solution of
\eqref{eq:SDE} itself, see Theorem~\ref{th:exist0}.

\begin{proof}
  This result follows essentially by combining
  Theorem~\ref{th:Epbound} and Lemma~\ref{lem:quadratic_variation}. We
  get under the same stopping time as before
  \begin{align*}
    \lnorm{Y_{\stop{t}}}^{p} &= \lnorm{Y_{0}}^{p}+\int_{0}^{\stop{t}} 
    (1+\sigma_{h}(s)) G_{1}(Y_{s})+(1-\sigma_{h}(s)) G_{2}(Y_{s}) \, ds
    +M_{\stop{t}},
  \end{align*}
  with $G_{1}$ and $G_{2}$ defined in \eqref{eq:Gdef}. The quadratic
  variation of the local martingale $M_{\stop{t}}$ can be estimated via
  Lemma~\ref{lem:quadratic_variation},
  \begin{align}
    \label{eq:Mqv}
    \Expect \, [M]_{\stop{t}}^{1/2} &\le 
    \Expect \int_{0}^{\stop{t}} C(1+\lnorm{Y_{s}}^{p}+
    \beta_{2}^{(0)}\lnorm{Y_{s}}^{p+1}) \, ds.
  \end{align}
  \textit{The case $\beta_{2}^{(0)} = 0$.} Using the bound
  \eqref{eq:momest} for the drift part as obtained in the proof of
  Theorem~\ref{th:Epbound} we get
  \begin{align*}
    \lnorm{Y_{\stop{t}}}^{p} &\le \lnorm{Y_{0}}^{p}+\int_{0}^{\stop{t}} 
    C(1+\lnorm{Y_{s}}^{p}) \, ds+|M_{\stop{t}}|, \\
    \intertext{combining this with \eqref{eq:Mqv} we obtain from
      Burkholder's inequality \cite[Chap.~IV.4]{protterSDE} that}
    \Expect \sup_{s \in [0,\stop{t}]} \lnorm{Y_{s}}^{p} &\le
    \lnorm{Y_{0}}^{p}+\int_{0}^{\stop{t}}
    C(1+\Expect\sup_{s' \in [0,s]} \lnorm{Y_{s'}}^{p}) \, ds.
  \end{align*}
  It follows that $\Expect \, \sup_{s \in [0,\stop{t}]}\lnorm{Y_{s}}^{p}$ is
  bounded in terms of the initial data and time $t$. Using Fatou's
  lemma the result follows by letting $\stopping \to \infty$.

  \textit{The case $\beta_{2}^{(0)} > 0$.} By Theorem~\ref{th:Epbound}
  we still have from \eqref{eq:Mqv} that
  \begin{align*}
    \Expect \, [M]_{\stop{t}}^{1/2} &\le
    \int_{0}^{\stop{t}} 
    C(1+\Expect\lnorm{Y_{s}}^{p+1}) \, ds 
    \le (e^{C\stop{t}}-1)(\lnorm{Y_{0}}^{p+1}+1).
  \end{align*}
  where we similarly obtain a bound in terms of the initial data
  $\|Y_{0}\|^{p+1}$.
\end{proof}

\subsection{Auxiliary lemmas}
\label{subsec:lemmas}

It is clear by now that the qualities of the kernel function
$\sigma_{h}(\cdot)$ will play a role in the behavior of the split-step
method. This motivates the following brief discussion.
\begin{lemma}
  \label{lem:sigma}
  Let $f: \Realdom \to \Realdom$ be a \cadlag\ piecewise constant
  function. Then
  \begin{align}
    \label{eq:sigma_lemma}
    \left| \int_{0}^{t} \sigma_{h}(s) f(s) \, ds \right| &\le
    \frac{h}{2} |f(t)|+
    \frac{h}{2} V_{[0,t]}(f),
  \end{align}
  where the total absolute variation may be exchanged with the square
  root of the quadratic variation $[f]_{t}^{1/2}$. Furthermore, if $t$
  is a multiple of $h$, then the first term on the right side of
  \eqref{eq:sigma_lemma} vanishes.
\end{lemma}

\begin{proof}
  Define
  \begin{align}
    \label{eq:Sigma_def}
    \Sigma_{h}(t) &\equiv \int_{0}^{t} \sigma_{h}(s) \, ds,
  \end{align}
  and observe that $|\Sigma_{h}(\cdot)| \le h/2$. Denote the left side
  of \eqref{eq:sigma_lemma} by $J$. Then with $(t_{k})_{k = 0}^{N}$
  the points of discontinuity of $f$ in $(0,t)$, but augmented with
  the two boundary points $\{0,t\}$, we obtain from summation by parts
  that
  \begin{align*}
    J &= \left| \sum_{k = 0}^{N-1} f(t_{k}) \; \Delta
      \Sigma_{h}(t_{k}) \right| = \left|f(t) \Sigma_{h}(t)-\sum_{k =
        0}^{N-1} \Delta f(t_{k}) \; \Sigma_{h}(t_{k+1}) \right|.
 \end{align*}
 The stated result now follows from the triangle inequality and, for
 the case of the quadratic variation, from the Cauchy-Schwartz
 inequality.
\end{proof}

\begin{lemma}
  \label{lem:sigma2}
  Let $G:\Realdom^{D} \to \Realdom$ be a globally Lipschitz continuous
  function with Lipschitz constant $L$ and let $f:\Realdom \to
  \Realdom^{D}$ be a piecewise constant \cadlag\ function. Then
  \begin{align}
    \label{eq:sigma_lemma2}
    \left| \int_{0}^{t} \sigma_{h}(s) G(f(s)) \, ds \right| &\le
    \frac{h}{2} |G(f(t))|+
    \frac{h}{2} L V_{[0,t]}(f),
  \end{align}
  with the same additional variants and simplifications as listed in
  Lemma~\ref{lem:sigma}.
\end{lemma}

\begin{proof}
  This follows because $g(t) := G(f(t))$ satisfies the requirements
  for $f$ in Lemma~\ref{lem:sigma}; clearly $|\Delta g(t_{k})| =
  |\Delta G(f(t_{k}))| \le L \|\Delta f(t_{k})\|$.
\end{proof}

\subsection{Strong convergence}
\label{subsec:convergence}

We are now ready to formulate and prove our main result of the paper.
\begin{theorem}[\textit{Strong convergence}]
  \label{th:converge}
  Let $X(t)$ and $Y(t)$ be solutions to \eqref{eq:Poissrepr} and
  \eqref{eq:splitSDE} with $X_{0} = Y_{0}$ and under
  Assumptions~\ref{ass:ass} and \ref{ass:split}, respectively. Then
  \begin{align}
    \label{eq:est0}
    \Expect\|(Y-X)(t)\|^{2} &\le h \, C_{t},
  \end{align}
  for $C_{t}$ some constant dependent on the final time $t$.
\end{theorem}

In the formulation above, the actual estimate that goes into
\eqref{eq:est0} is elaborated upon in \eqref{eq:est1} below. Also, for
brevity and inspired by most actual use, we only consider the case of
deterministic initial data.

\begin{proof}
  Under the stated assumptions both processes are well behaved
  (Theorem~\ref{th:exist0} and \ref{th:exist}), so the strategy of
  proof is to define a suitable stopping time $\taustopping$ such that
  the probability that $t \ge \taustopping$ can be made arbitrarily
  small. We put
  \begin{align}
    \nonumber
    \taustopping &:= 
    \inf_{t \ge 0}\{\lnorm{X_{t}} \vee \lnorm{Y_{t}} > \stopping\}, \\
    \intertext{and as before $\stop{t} = t \wedge \taustopping$.
      Subtracting \eqref{eq:Poissrepr} from \eqref{eq:splitSDE} we
      get}
    \nonumber
    (Y-X)(\stop{t}) &= -\sum_{r \in \R_{1},\R_{2}} \stoich_{r} \Biggl[
      \Pi_{r} \left(
      \int_{0}^{\stop{t}} (1 \pm \sigma_{h}(s))w_{r}(Y_{s-}) \, ds \right)- \\
      \label{eq:startdiff}
      &\hphantom{= -\sum_{r \in \R_{1},\R_{2}} \stoich_{r} \Biggl[}
      \Pi_{r} \left(
      \int_{0}^{\stop{t}} w_{r}(X_{s-}) \, ds \right) \Biggr],
  \end{align}
  where the sign is to be chosen in accordance with
  \eqref{eq:split_sum}.

  Under the stopping time and using the local Lipschitz condition we
  first produce the basic estimates
  \begin{align*}
    &\int_{0}^{\stop{t}} w_{r}(X_{s-}) \, ds \le 
    \int_{0}^{\stop{t}} w_{r}(0)+P L_{r} \, ds \le C_0^{(r)}(P) \stop{t}, \\
    &\int_{0}^{\stop{t}} (1 \pm \sigma_{h}(s))w_{r}(Y_{s-}) \, ds 
    \le 2C_0^{(r)}(P) \stop{t},
  \end{align*}
  with, for brevity, $L_{r} \equiv L_{r}(P)$, and using that $|1 \pm
  \sigma_{h}| = 2$.

  Taking the square Euclidean norm and expectation value of
  \eqref{eq:startdiff} we find by Lemma~\ref{lem:Stopping} in the form
  of \eqref{eq:finalStop} using the above basic bounds that
  \begin{align}
    \Expect \| (Y-X)(\stop{t})\|^{2} \le \|\stoich\|^{2} \sum_{r=1}^{R}
    (4C_0^{(r)} \stop{t}+1) \Expect [A_r],
  \end{align}
  in terms of
  \begin{align}
    \nonumber
    A_{r} &= \left| \int_{0}^{\stop{t}} w_{r}(Y_{s-})-w_{r}(X_{s-}) \,
      ds \pm \int_{0}^{\stop{t}} \sigma_{h}(s)w_{r}(Y_{s-}) \, ds \right|. \\
    \intertext{Using Assumption~\ref{ass:ass}~\eqref{it:lip} anew we get}
    \nonumber
    A_{r} &\le \int_{0}^{\stop{t}} L_{r} \|(Y-X)(s-)\| \,
    ds +\underbrace{\left| \int_{0}^{\stop{t}} \sigma_{h}(s)w_{r}(Y_{s-}) \,
        ds \right|}_{=: B_{r}}. \\
    \intertext{Also, by Lemma~\ref{lem:sigma2} we obtain}
    \nonumber
    B_{r} &\le \frac{h}{2} |w_{r}(Y_{\stop{t}-})|+
    \frac{h}{2} L_{r} [Y]_{\stop{t}-}^{1/2}. \\
    \intertext{For the quadratic variation we note that, since
      $\|\cdot\| \le \|\cdot\|_{1} \le \lnorm{\cdot}$, we have that
      $[Y]_{t} \le [\lnorm{Y}]_{t}$, where the case $p = 1$ of
      Lemma~\ref{lem:quadratic_variation} applies. Estimating using
      Theorem~\ref{th:Epbound} we get after some work,}
    \label{eq:Brbnd}
    \Expect[B_{r}] &\le C_{1}^{(r)}(P)h+C_{2}^{(r)}(P)h \stop{t} \, \exp(C\stop{t}),
  \end{align}
  where the constants do not depend on $h$, and where $C_{1}^{(r)}$
  may be taken as zero provided that $\stop{t}$ is a multiple of $h$.

  Summing this over $r$ we get
  \begin{align*}
    \Expect \| (Y-X)(\stop{t})\|^{2} &\le
    C_{1}h+C_{2}h \stop{t} \, \exp(C\stop{t})+ \\
    &\hphantom{\le}
    \underbrace{\|\stoich\|^{2} \sum_{r = 1}^{R} 
    (4C_0^{(r)}\stop{t}+1)L_{r}}_{=: L(\stop{t}) = L(\stop{t},P)}
    \Expect \left[ \int_{0}^{\stop{t}} \|(Y-X)(s)\|^{2} \,
      ds \right],
    \intertext{where the ``integer inequality'', $n \le n^{2}$ for $n
      \in \Intdom$, was used. By Grönwall's inequality this implies
      the bound}
    \Expect \| (Y-X)(\stop{t})\|^{2} &\le 
    h \left[ C_{1}+C_{2} \stop{t} \, \exp(C\stop{t}) \right]
    \exp(\stop{t}L(\stop{t})).
  \end{align*}

  Using brackets to denote characteristic functions we write
  \begin{align*}
    \Expect\|(Y-X)(t)\|^{2} &= \Expect \left[ \|(Y-X)(t)\|^{2}[t
      < \taustopping]\right]+
    \underbrace{\Expect \left[ \|(Y-X)(t)\|^{2}
        [t \ge \taustopping] \right]}_{=: M}. \\
    \intertext{To bound $M$, note first that by Cauchy-Schwartz's inequality,}
    M &\le \left(
        \Expect\|(Y-X)(t)\|^{4} \right)^{1/2} \left( \Prob[t \ge
        \taustopping] \right)^{1/2}. \\
    \intertext{Since $\Prob[t \ge \taustopping] = \Prob[\sup_{s
          \in [0,t]} \lnorm{X_{s}} \vee \lnorm{Y_{s}} > \stopping]$ we
      get from Markov's inequality}
    M &\le \left(
    \Expect\|(Y-X)(t)\|^{4} \right)^{1/2} P^{-1/2} \left( \Expect \sup_{s
      \in [0,t]} \lnorm{X_{s}} \vee \lnorm{Y_{s}} \right) ^{1/2}.
  \end{align*}
  By Theorem~\ref{th:exist0} and \ref{th:exist} we find that $M$
  converges to zero as $P \to \infty$. In particular, for any given
  $\varepsilon \in (0,1)$, we can find $P$ large enough (and hence
  also $C_{1}$, $C_{2}$, and $L(t)$) such that 
  \begin{align}
    \nonumber
    M &\le \varepsilon \Expect\|(Y-X)(t)\|^{2}.
    \intertext{Combining we thus get}
    \label{eq:est1}
    \Expect\|(Y-X)(t)\|^{2} &\le \frac{h}{1-\varepsilon}
    \left[ C_{1}+C_{2} t \, \exp(Ct) \right]
    \exp(t L(t)),
  \end{align}
  where $C_{1}$, $C_{2}$, and $L(t)$ depend on $\varepsilon$ via the
  choice of $P$.
\end{proof}

We now offer some brief comments on this result. The estimated error
\emph{growth} in any given interval of time $[0,t]$ is clearly quite
pessimistic, but is on the other hand also quite general, relying as
it does mainly on the existence of local Lipschitz constants. Also,
the fact that the constant $C_{1}$ may be disregarded when $t$ is a
multiple of the split-step $h$, is perhaps mostly of theoretical
interest as the other terms seem to dominate by far. One may simply
think of this as a smaller error when the split-step solution is
viewed at the discrete mesh $[0,h,2h,\ldots]$.

Finally we comment also that Theorem~\ref{th:converge} may be
strengthen in the direction of convergence in
\begin{align}
  \Expect \sup_{s \in [0,t]} \|(Y-X)(t)\|^{2},
\end{align}
by bounding the martingale part and relying on Burkholder's
inequality. Since this is somewhat tedious and does not add much
insight we have chosen to omit it altogether.

% \subsection{Parameter sensitivity}
%
% \Expect \|(X^{(\theta)}-X)(t)\| \le 
% \Expect\|(Y^{(\theta)}-Y+(X^{(\theta)}-Y^{(\theta)})-(X-Y))(t)\| \le
% \Expect\|(Y^{(\theta)}-Y)(t)\|+
% \Expect\|(X^{(\theta)}-Y^{(\theta)})(t)\|+
% \Expect\|(X-Y)(t)\| \le C\theta^{1/2}+Ch^{1/2}

%**************************************************************************

\section{Numerical examples}
\label{sec:experiments}

To illustrate the theory developed in the previous sections, but also
as a means to investigate the sharpness of some of the bounds, a few
selected numerical examples are presented here. In
\S\ref{subsec:implementation} we briefly discuss the actual
implementation of the method implied by \eqref{eq:splitSDE}, and in
\S\ref{subsec:bd} and \S\ref{subsec:dimer} the convergence of the
split-step method applied to two one-dimensional examples is
investigated. Concretely, our computational analysis investigates the
convergence dependency in the presence of non-linearities, the weak
convergence behavior, and the qualities of higher order split-step
methods. In \S\ref{subsec:bimol} convergence over longer
time-intervals is discussed and in \S\ref{subsec:illposed}, finally,
we look at a split which violates Assumption~\ref{ass:split}.

\subsection{Implementation}
\label{subsec:implementation}

The idea to use the operational time representation
\eqref{eq:Poissrepr} to devise computational algorithms has been
employed previously for time-parallel algorithms \cite{ssa_parareal}
and for parameter sensitivity computations
\cite{aem_proceeding,sensitivitySSA}. The \emph{Common Reaction Path
  Method} \cite{sensitivitySSA} simulates \eqref{eq:Poissrepr} using
$R$ separate streams of random numbers and can be regarded as a
careful implementation of the \emph{Next Reaction Method}
\cite{gillespiemodern} such that a consistent operational time is
always defined. An additional improvement reported for spatial
problems in \cite{aem_proceeding} is to handle rates that become zero
in a somewhat careful way. By explicitly storing the operational time
$\tau^{\mathrm{old}}$ and associated rate $w^{\mathrm{old}}$
\emph{before} the rate vanished we can ``activate'' the reaction
channel by using the rescaling technique \cite{gillespiemodern},
\begin{align}
  \tau^{\mathrm{new}} = t_{\mathrm{current}}+
  \left(\tau^{\mathrm{old}}-t_{\mathrm{current}}\right)
  w^{\mathrm{old}}/w^{\mathrm{new}},
\end{align}
where $w^{\mathrm{new}}$ is the first non-zero rate encountered. As
opposed to drawing a new random number whenever the channel is
reactivated, this technique preserves the consistency of the current
operational time.

All these implementation techniques transfer nicely to the split-step
formulation \eqref{eq:splitSDE}. The kernel function $\sigma_{h}$ may
be thought of as generating deterministic events at points in time
which are multiples of $h/2$, where the active and passive pathways
simply change roles. Besides being able to compare trajectories
path-wise for different values of $h$, a great feature with this
implementation strategy is of course that the limit $h \to 0$ is
directly computable.

For the computations presented below we used the estimator
\begin{align}
  \label{eq:mean}
  \Expect[(Y-X)(t)]^{2} &\approx M \equiv 
  \frac{1}{N}\sum_{i = 1}^{N} (Y-X)(t; \Probelem_{i})^{2},
  \intertext{where $\Probelem_{i}$ indicate independent trajectories
    coupled and computed according to the previous description. To
    estimate the uncertainty in \eqref{eq:mean} we compute}
  \label{eq:std}
  S^{2} &\equiv \frac{1}{N-1}\sum_{i = 1}^{N}
  \left[(Y-X)(t; \Probelem_{i})^{2}-M \right]^{2},
\end{align}
and use a suitable multiple of $S/\sqrt{N}$ as a measure of the
uncertainty.

\subsection{Linear birth-death process}
\label{subsec:bd}

We first consider the model
\begin{align}
  \label{eq:bd}
  &\emptyset \xrightleftharpoons[\mu X]{k} X,
\end{align}
with $[k,\mu] = [5,0.05]$ for time $t \in [0,100]$ and with $X(0) =
50$. This model approaches a steady-state Poissonian distribution
around the mean value 100 and executes about $10^{3}$ events in the
given time interval. Sample illustrations of this process are
displayed in Figure~\ref{fig:bdsample}.

In the notation of Assumption~\ref{ass:ass} the birth-death model is
characterized by $\alpha = -\mu$, $\beta_{1} = \mu/2$, and $\beta_{2}
= 0$. By the negative sign of $\alpha$, the dynamics is dissipative
for states $x > k/\mu$ and we note also that the propensities are
globally Lipschitz with constant $L = \mu$. Computational results are
reported in comparison with the dimerization model, next to be
introduced.

\begin{figure}[htp]
  \includegraphics{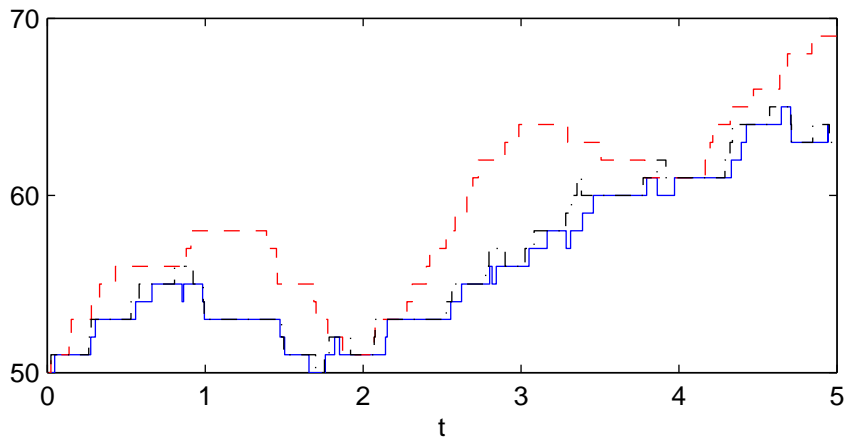}
  \caption{Sample trajectories of the linear birth-death
    model. \textit{Solid:} direct simulation, \textit{dashed:}
    split-step solution ($h = 2$), \textit{dash-dot:} split-step
    solution ($h = 1/4$). For $h = 2$ the periodic effect of the
    kernel function $\sigma_{h}$ is clearly visible.}
  \label{fig:bdsample}
\end{figure}

\subsection{Dimerization}
\label{subsec:dimer}

As a simple nonlinear extension of \eqref{eq:bd} we consider
\begin{align}
  \label{eq:dim}
  &\emptyset \xrightarrow{k} X, \quad X+X \xrightarrow{\nu X(X-1)} \emptyset.
\end{align}
We normalize the model such that the equilibrium mean
$\sqrt{k/(2\nu)}$ coincide with that of \eqref{eq:bd} and we also use
the same birth constant $k$. This means that for both models, about
the same number of events is generated in the considered interval of
time.

It is instructive to try to reason about the effect of the
non-linearity in \eqref{eq:dim} in comparison with the fully linear
model \eqref{eq:bd}. Although by necessity $\alpha = 0$ in
Assumption~\ref{ass:ass}~\eqref{it:1bnd}, it still holds true that the
dynamics is dissipative for states larger than the equilibrium mean
value. Also, due to the quadratic reaction, there is no longer a
global Lipschitz constant in Assumption~\ref{ass:ass}~\eqref{it:lip},
but one may note that close to the equilibrium mean value, $L \sim \nu
\times (x-1) \sim \mu^{2}/(2k) \times k/\mu = \mu/2$. Since the
quadratic reaction involves two molecules one may argue that the
effective Lipschitz constant for \eqref{eq:dim} approximately matches
that of \eqref{eq:bd}.

A more striking difference between the two models can be found in
Assumption~\ref{ass:ass}~\eqref{it:bnd}. Namely, for \eqref{eq:dim}
the left side reads
\begin{align}
  \label{eq:dimbnd}
  k/2 \left[1+4\nu/k \times x(x-1) \right] &\sim 3k/2, \\
  \intertext{whereas for \eqref{eq:bd} we have}
  \label{eq:bdbnd}
  k/2 \left[1+\mu/k \times x\right] &\sim k,
\end{align}
again, assuming that $x$ is close to the equilibrium mean value.

\subsubsection{Strong convergence}

We first consider the strong convergence of the split-step method. The
mean square error as a function of the split-step $h$ is shown in
Figure~\ref{fig:bd_dim_convergence} and the results show that,
although both the birth-death model \eqref{eq:bd} and the dimerization
model \eqref{eq:dim} are unbounded problems, the strong order is still
$1/2$ as predicted by Theorem~\ref{th:converge}.

The interesting observation to be made is that, with the exception of
the case $h = 1$, the mean square error for the dimerization problem
is consistently almost a factor of two larger than for the birth-death
problem (the measured factor falls in the range $[1.5,2.1]$ for the
cases studied). It is not easy to strictly analyze the reasons behind
this phenomenon but we may argue heuristically as follows.

Let us take \eqref{eq:est1} in Theorem~\ref{th:converge} for some
fixed value of $\varepsilon$ as an estimate of the error. By the
set-up of the measurements, $C_{1} = 0$ in \eqref{eq:est1}, and we
have also argued previously that the effective Lipschitz constants for
the two models are about the same. The difference between the two
models has therefore been isolated to the expression $C_{2}ht
\exp(Ct)$ in the right side of \eqref{eq:est1}, which can be traced
back to the bound of $B_{r}$ in \eqref{eq:Brbnd}. In turn, this
estimate comes from the bound on the quadratic variation in
Lemma~\ref{lem:quadratic_variation} and relies indirectly also on the
moment estimate from Theorem~\ref{th:Epbound}. In the present case the
first few moments of \eqref{eq:bd} and \eqref{eq:dim} are of
comparable magnitude so we thus focus our attention to the constant
$C$ in \eqref{eq:quadratic_variation} of
Lemma~\ref{lem:quadratic_variation}. In the proof, this constant
emerges in \eqref{eq:special} and is a direct consequence of
Assumption~\ref{ass:ass}~\eqref{it:bnd}.

This assumption was investigated in
\eqref{eq:dimbnd}--\eqref{eq:bdbnd} above where we found a difference
between the two models in the form of an overall factor of $3/2$ and a
factor of $2$ when considering the non-constant part, in striking
agreement with the measured factor $\in [1.5,2.1]$.

It is an interesting and challenging question if the above heuristic
way of reasoning can somehow be put on firm grounds. In particular, it
would be very useful if the use of `effective' Lipschitz constants
could render a consistent analysis.

\begin{figure}[htp]
  \includegraphics{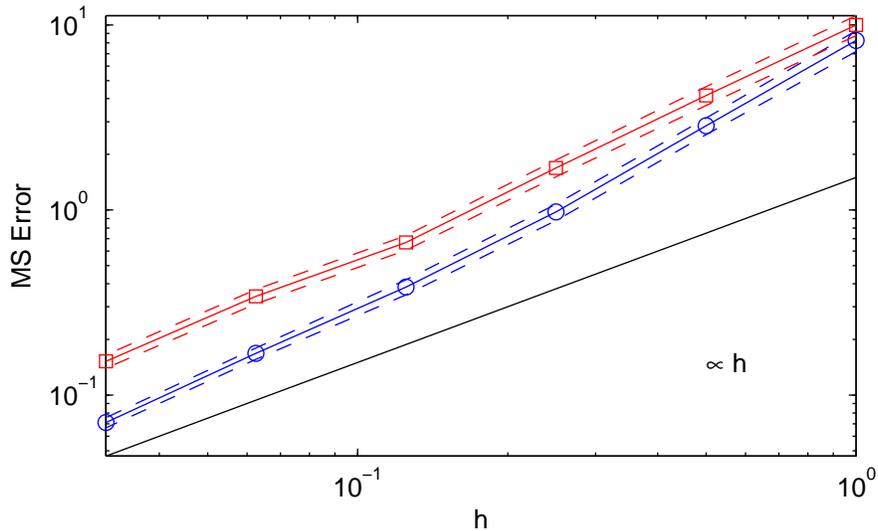}
  \caption{Convergence in mean square of the split-step method with
    decreasing $h$ for the linear birth-death model \eqref{eq:bd}
    (\textit{circles}) and for the dimerization \eqref{eq:dim}
    (\textit{squares}). Dashed lines indicate the estimator's
    uncertainty $\pm S/\sqrt{N}$ (cf.~\eqref{eq:std}). For both
    cases the method converges faster initially but approaches the
    strong order $1/2$ as $h$ becomes smaller.}
  \label{fig:bd_dim_convergence}
\end{figure}

\subsubsection{Weak convergence}

We briefly look also at the weak convergence of the split-step
scheme. Hence we estimate for varying split-step $h$,
\begin{align}
  \label{eq:weak}
  &\left| \Expect[f(Y_{t})]-\Expect[f(X_{t})] \right|,
  \intertext{with $f$ some smooth function and otherwise following the
    computational procedure in \S\ref{subsec:implementation}. To be
    concrete we took}
  &f_{1}(x) = x, \\
  &f_{2}(x) = x(x-1),
\end{align}
that is, the first two factorial moments. In analogy with
Figure~\ref{fig:bd_dim_convergence}, in
Figure~\ref{fig:bd_dim_wconvergence} we report the error
\eqref{eq:weak} for these two cases. As expected we find a first order
weak error when measured in the split-step $h$. Perhaps more
interesting is the observation that the errors for the two models are
very similar and that there is no visible impact from the
non-linearity.

\begin{figure}[htp]
  \includegraphics{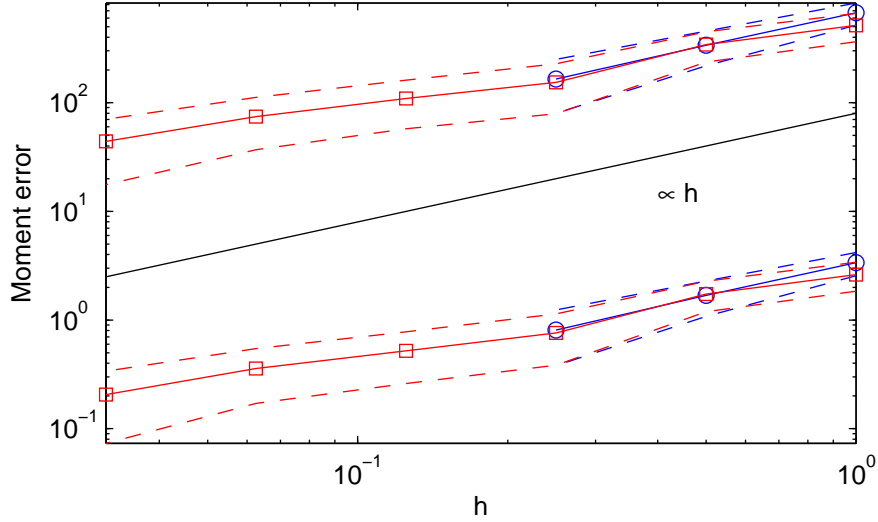}
  \caption{Weak convergence of the split-step method with decreasing
    $h$. \textit{Circles}: linear birth-death model \eqref{eq:bd} and
    \textit{squares}: dimerization \eqref{eq:dim}. \textit{Top:}
    second order factorial moment, \textit{bottom:} first order
    factorial moment. For the birth-death model, only data for the
    larger values of $h$ are available as the variance was too large
    to determine the sign of the error within the target uncertainty.}
  \label{fig:bd_dim_wconvergence}
\end{figure}

\subsubsection{The 2nd order Strang split}
\label{subsubsec:Strang}

The second order Strang split is traditionally written in operator
form in the style of \eqref{eq:CMEsplit1}--\eqref{eq:CMEsplit2},
{\small
\begin{align}
  \label{eq:CMEsplit1_Strang}
  \tilde{p}_h(x,t+h/2) &= p_h(x,t)+\int_{t}^{t+h/2} 
  \sum_{r \in \R_{1}} w_{r}(x+\stoich_{r})\tilde{p}_h(x+\stoich_{r},s)-
  w_{r}(x)\tilde{p}_h(x,s) \, ds, \\
  \label{eq:CMEsplit2_Strang}
  \tilde{\tilde{p}}_h(x,t+h) &= \tilde{p}_h(x,t+h/2)+\int_{t}^{t+h} 
  \sum_{r \in \R_{2}} w_{r}(x+\stoich_{r})\tilde{\tilde{p}}_h(x+\stoich_{r},s)-
  w_{r}(x)\tilde{\tilde{p}}_h(x,s) \, ds, \\
  \label{eq:CMEsplit3_Strang}
  p_h(x,t+h) &= \tilde{\tilde{p}}_h(x,t+h)+\int_{t+h/2}^{t+h} 
  \sum_{r \in \R_{1}} w_{r}(x+\stoich_{r})p_h(x+\stoich_{r},s)-
  w_{r}(x)p_h(x,s) \, ds,
\end{align}}
which via two intermediate steps takes us from time $t$ to $t+h$. A
moments consideration gives that this is just the same thing as
substituting the kernel function $\sigma_{h}(s)$ in
\eqref{eq:splitSDE} with the time shifted version
$\sigma_{h}(s+h/4)$. Importantly, the resulting compact notation is
open to the same kind of analysis performed in \S\ref{sec:analysis}
and we draw the conclusion that also this method can be expected to
converge at strong order $1/2$. We should mention though, that to
strictly prove that the convergence order is not actually higher might
require some more work.

A feature with the split method
\eqref{eq:CMEsplit1_Strang}--\eqref{eq:CMEsplit3_Strang} is that it
can be expected to be second order weakly convergent. This follows
heuristically from the fact that it is a second order operator
splitting method in the finite-dimensional case, and hence can be
expected to perform as such also in cases that are of effectively
bounded character.

In Figure~\ref{fig:bd_dim_convergence2} we display the mean square
strong error as a function of the split-step $h$ for the two models
considered in this section. The asymptotic behavior is quite similar
to Figure~\ref{fig:bd_dim_convergence} in that we still approach the
strong order $1/2$ as $h \to 0$ and in that the error reduction factor
between the two models is about the same, or perhaps slightly larger
$\in [1.7,3.2]$.

By far the most striking difference is that the mean square error is
now between 1.5 to 2 \emph{orders} smaller than before. It is
unfortunately quite difficult to explain this in the setting of
\S\ref{subsec:convergence} since the analysis there would be quite
similar upon substituting $\sigma_{h}(s) \mapsto
\sigma_{h}(s+h/4)$. In particular, assuming the Strang split to be
second order weakly convergent, it is difficult to see where this
feature could enter the analysis effectively.

\begin{figure}[htp]
  \includegraphics{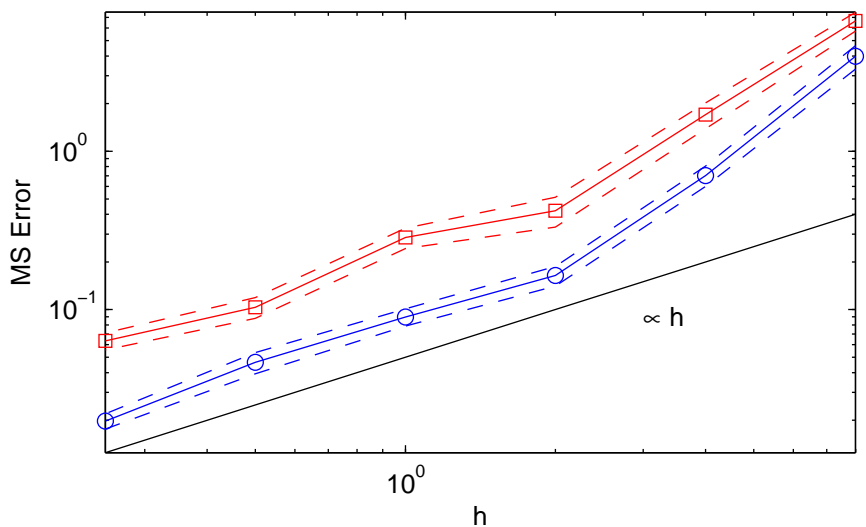}
  \caption{Convergence in mean square of the Strang split method
    \eqref{eq:CMEsplit1_Strang}--\eqref{eq:CMEsplit3_Strang}, following
    closely the notation in
    Figure~\ref{fig:bd_dim_convergence}. Although the asymptotic
    strong order of convergence is still $1/2$, the error is
    considerably smaller for the same values of $h$.}
  \label{fig:bd_dim_convergence2}
\end{figure}

\subsection{Elementary bimolecular reaction, case of no equilibrium}
\label{subsec:bimol}

As a slightly more involved model we consider the following
bimolecular birth-death system,
\begin{align}
  \label{eq:bimol}
  \left. \begin{array}{c}
    \emptyset \xrightarrow{k_{1}} X \\
    \emptyset \xrightarrow{k_{1}} Y \\
    X+Y \xrightarrow{k_{2} XY} \emptyset
  \end{array} \right\},
\end{align}
with the stoichiometric matrix
\begin{align}
  \stoich &= \begin{bmatrix*}[r]
    -1 & 0 & 1 \\
    0 & -1 & 1
  \end{bmatrix*}
\end{align}
and reaction propensities $w(x) = [k_{1},k_{1},k_{2}x_{1}x_{2}]^{T}$
for $x = [x_{1},x_{2}]^{T}$.

To get some feeling for the behavior of \eqref{eq:bimol}, we define
the difference process $U(t) = (X-Y)(t) \in \Intdom$. From It\^o's
formula \eqref{eq:Ito} with $f(x) = x_{1}-x_{2}$ we get
\begin{align}
  dU_{t} &= df(X_{t}) = -[-1,1,0] \, \fatmu(dt),
\end{align}
which is equivalent to the model
\begin{align}
  &\emptyset \xrightleftharpoons[k_{1}]{k_{1}} U,
\end{align}
which is just a constant intensity random walk on all of
$\Intdom$. One can solve this explicitly in terms of two independent
Poisson distributions,
\begin{align}
  U_{t} &= U_{0}+\Pi_{1}(k_{1}t)-\Pi_{2}(k_{1}t)
  \sim U_{0}+\mathcal{N}(U_{0},2k_{1}t), \quad \mbox{$t \to \infty$,}
\end{align}
for $\mathcal{N}$ a Gaussian random variable of the indicated mean and
variance. Hence the system takes longer and longer excursions away
from the origin and we have an example of an equilibrium density which
clearly does not exist.

We split the problem \eqref{eq:bimol} with $\R_{1} = \{1,2\}$ and
$\R_{2} = \{3\}$, that is, with the two birth processes evolved
simultaneously. In Figure~\ref{fig:bimolconvergence} we display the
mean square error as a function of time for several values of the
split-step parameter $h$. Experimentally we find that the strong order
of convergence is still $1/2$ even for quite long simulation times and
despite the fact that the behavior of the underlying process is open,
visiting states further and further away.

\begin{figure}[htp]
  \includegraphics{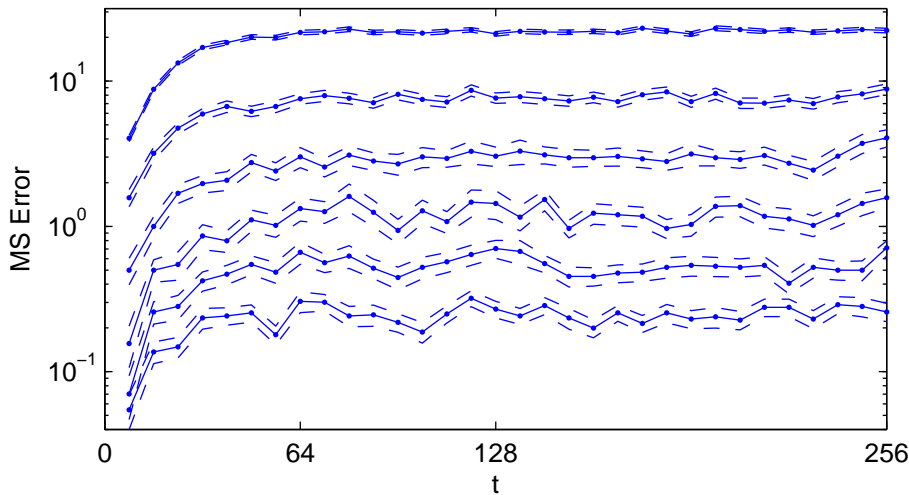}
  \caption{Mean square error of the split-step method for the
    bimolecular problem \eqref{eq:bimol} as a function of time
    $t$. From top to bottom: $h = [4,2,1,1/2,1/4,1/8]$. The asymptotic
    strong order of the method is $q \approx
    [0.76,0.67,0.67,0.61,0.55]$ for the cases shown in the figure and
    averaging over the data for $t \in [128,256]$.}
  \label{fig:bimolconvergence}
\end{figure}

\subsection{An ill-posed split}
\label{subsec:illposed}

In this section we apply the split-step method to a model which does
not comply with Assumption~\ref{ass:split}. This is quite challenging
computationally and thus the model settled for was selected after
trying several quite different cases. Although somewhat artificially
looking, the rather simple system chosen was
\begin{align}
  \label{eq:illposed}
  \left. \begin{array}{ll}
      \emptyset \xrightarrow{10} X & 3X \xrightarrow{x(x-1)(x-2)/2} X \\
      3X \xrightarrow{x(x-1)(x-2)/2} X  & 3X \xrightarrow{x(x-1)(x-2)} 4X
  \end{array} \right\}.
\end{align}
The stoichiometric vector is given by $\stoich = [-1,-2,-2,1]$ and
hence the problem is very strongly dissipative for states large enough
in view of the fact that the ingoing \emph{cubic} reactions
dominates. However, by splitting the system according to $\R_{1} =
\{1,2\}$ and $\R_{2} = \{3,4\}$, we have that the second pair violates
Assumption~\ref{ass:split}. In fact, this sub-system can be shown to
explode in the second moment for $t \lesssim X(0)^{-3}/3$ whenever
$X(0) \ge 3$ despite the fact that the net drift is zero
\cite[Proposition~4.1]{jsdestab}.

In the simulation we worked with the stopped process $\stop{X}(t) =
X(t) \wedge \stopping$, where $\stopping = 10^{3}$ was used and where
$t \in [0,1]$ with $X(0) = 10$ was considered. In
Figure~\ref{fig:illposeddivergence} we report the results of a
convergence study for a selection of comparably small split-step sizes
$h$. The convergence behavior is rather intriguing, with an initial
much slower strong order convergence rate of about $1/4$, but which
picks up speed for $h$ smaller than about $10^{-3}/3$, which is also
approximately the maximum time interval over which the second
sub-system does not explode in variance. Although the method still
seems to converge, the effects of choosing a split containing an
ill-posed sub-system are clearly visible.

\begin{figure}[htp]
  \includegraphics{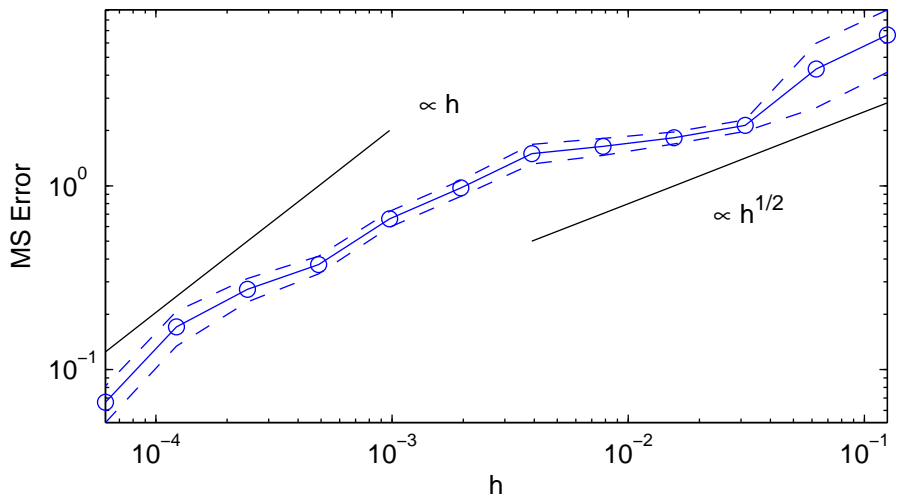}
  \caption{Mean square convergence for the split-step method applied
    to the problem \eqref{eq:illposed}, for which one of the
    sub-systems is ill-posed in the sense of unbounded second moments
    in finite time. The order of convergence is roughly $1/4$ for the
    range of $h$ studied here but appears to improve somewhat towards
    the smaller values.}
  \label{fig:illposeddivergence}
\end{figure}

%**************************************************************************

\section{Conclusions}
\label{sec:conclusions}

We have proposed a framework for analyzing certain popular split-step
methods for jump stochastic differential equations. The framework
consists of a formulation of the methods in operational time via the
split-step kernel function, enabling a meaningful path-wise coupling
even in the limit $h \to 0$. We have also presented concrete
assumptions and \textit{a priori} results which together with the
theoretical convergence results form a basis for the sound use and
further development of these types of methods.

In performing the computational experiments we also developed an actual
implementation of our formulation which is useful in assessing the
split-step solution quality as a function of various parameters. For
example, this is important when experimentally approaching set-ups for
which the proposed sufficient conditions for convergence are violated.

The numerical experiments illustrate the theory in different ways and
also open up some intriguing questions. For instance, how can the
surprisingly good relative efficiency of the second order Strang
method versus the simpler first order method best be explained in the
setting of strong convergence? We also observed a better convergence
over longer time intervals than perhaps intuitively expected, and we
noted that the method converges even in a case for which one of the
split sub-systems was ill-posed in the sense of finite-time moment
explosions.

More practically, an important and natural extension is when the
different sub-systems are evolved by specially designed approximative
methods, a quite common approach in multiscale and multiphysics
problems. The proposed framework should remain a viable approach as
long as these methods may also be analyzed in a path-wise sense, a
point to which we intend to return to in future work.

%**************************************************************************

\section*{Acknowledgment}

The writing of this paper was initiated at the inspiring BIRS workshop
\textit{``Particle-Based Stochastic Reaction-Diffusion Models in
  Biology''}, in Banff, Alberta, Canada, in November 2014. The author
also likes to acknowledge several fruitful discussions with Augustin
Chevallier as well as constructive referee comments.

The work was supported by the Swedish Research Council within the
UPMARC Linnaeus center of Excellence.

%**************************************************************************

\newcommand{\doi}[1]{\href{http://dx.doi.org/#1}{doi:#1}}
\newcommand{\available}[1]{Available at \url{#1}}
\newcommand{\availablet}[2]{Available at \href{#1}{#2}}

\bibliographystyle{abbrvnat}
\bibliography{../../../stefan}

\end{document}